\documentclass{amsart}
\usepackage{amsmath,amsthm}
\usepackage{amsfonts,amssymb}

\usepackage{enumerate}

\newtheorem{thm}{Theorem}[section]
\newtheorem{cor}[thm]{Corollary}
\newtheorem{lem}[thm]{Lemma}
\newtheorem{prop}[thm]{Proposition}

\newtheorem{defn}[thm]{Definition}
\theoremstyle{remark}

 \def\a{{\alpha}}
 \def\b{{\beta}}
 \def\g{{\gamma}}
 \def\k{{\kappa}}
 
 \def\l{{\lambda}}

 \def\s{{\sigma}}
 \def\la{{\langle}}
 \def\ra{{\rangle}}

 \def\sB{{\mathsf B}}
 \def\sC{{\mathsf C}}
 
 \def\sH{{\mathsf H}}
 \def\sK{{\mathsf K}}
 
 \def\sP{{\mathsf P}}
 \def\sQ{{\mathsf Q}}
 \def\sU{{\mathsf U}}

 \def\xb{{\mathbf x}}
 \def\yb{{\mathbf y}}

 \def\CF{{\mathcal F}}

 \def\CE{{\mathcal E}}

 \def\CV{{\mathcal V}}

 \def\NN{{\mathbb N}}

 \def\RR{{\mathbb R}}
 \def\ZZ{{\mathbb Z}}
        
        \def\dim{\operatorname{dim}}
        
\def\one{{\mathbf{1}}}

 \def\brho{{\boldsymbol{\rho}}}

\def\f{\frac}

\begin{document}
 
\title[Hahn, Jacobi, and Krawtchouck polynomials of several variables]
{Hahn, Jacobi, and Krawtchouck polynomials of several variables}
\author{Yuan Xu}
\address{Department of Mathematics\\ University of Oregon\\
    Eugene, Oregon 97403-1222.}
\email{yuan@math.uoregon.edu}

\date{\today}
\thanks{The work was supported in part by NSF Grant DMS-1106113}
\keywords{ Hahn, Jacobi, Krawtchouck, polynomials, several variables, reproducing kernels, Poisson kernels}
\subjclass[2000]{33C50, 33C70}
                          
\begin{abstract}
Hahn polynomials of several variables can be defined by using the Jacobi polynomials on the simplex 
as a generating function. Starting from this connection, a number of properties for these two families
of orthogonal polynomials are derived. It is shown that the Hahn polynomials
appear as connecting coefficients between several families of orthogonal polynomials on 
the simplex. Closed-form formulas are derived for the reproducing kernels of the Hahn polynomials
and Krawtchouck polynomials. As an application, the Poisson kernels for the Hahn polynomials and 
the Krawtchouck polynomials are shown to be nonnegative. 
\end{abstract} 

\maketitle                      
 
\section{Introduction}
\setcounter{equation}{0}

Let $N$ be a positive integer. For $\k \in \RR^{d+1}$ with $\k_1 > -1,\ldots,  \k_{d+1} > -1$ and $x \in \RR^d$, 
the Hahn weight function of $d$ variables is defined by 
\begin{equation}\label{eq:weightH}
   \sH_{\kappa,N}(x): =  \prod_{i=1}^d  \binom{\k_i+1}{x_i}
         \binom{\k_{d+1}+1}{N-|x|}, \quad x \in \NN_0^d, \quad |x| \le N,
\end{equation}
where $|\k|:= \k_1+\cdots + \k_{d+1}$ and $|x| := x_1+ \cdots+ x_d$, and the Hahn polynomials of $d$-variables are discrete orthogonal polynomials with respect to the inner product 
\begin{equation}\label{eq:ipdH}
    \la f, g \ra_{\sH_{\k,N}}: = \frac{1}{\binom{|\k|+N}{N}} \sum_{x \in \NN_0^d: |x| \le N} f(x) g(x) \sH_{\k,N}(x).
\end{equation}
For $\rho \in \RR^d$ with $0 < \rho_i <1$, 
$1\le i \le d$ and $|\rho| < 1$, the Krawtchouck weight function of $d$-variables 
is defined by 
\begin{equation}\label{eq:weightK}
   \sK_{\rho,N}(x): =   \prod_{i=1}^{d} \frac{\rho_i^{x_i}}{x_i!} \frac{(1-|\rho|)^{N-|x|}}{(N-|x|)!},
   \quad x \in \NN_0^d, \quad |x| \le N,
\end{equation}
and the Krawtchouck polynomials of several variables are discrete orthogonal polynomials with respect to 
the inner product 
\begin{equation}\label{eq:ipdK}
    \la f, g \ra_{\sK_{\rho, N}}: =  \frac{1}{N!} \sum_{x \in \NN_0^d: |x| \le N} f(x) g(x) \sK_{\rho, N}(x). 
\end{equation}
The Jacobi weight function on the simplex $T^d: = \{x \in \RR^d: x_i \ge 0, |x|\le 1\}$ is defined by 
\begin{equation}\label{eq:weightW}
  W_\kappa (x):=  x_1^{\kappa_1} \cdots x_d^{\kappa_d} (1-|x|)^{\kappa_{d+1}},  \quad x \in T^d,
\end{equation}
where $\k \in \RR^{d+1}$ with $\k_1 > -1,\ldots, \k_{d+1} > -1$. The Jacobi polynomials of $d$-variables 
are orthogonal polynomials with respect to the inner product
\begin{equation}\label{eq:innerW}
   \la f, g \ra_{W_\k} : = \frac{\Gamma(|\kappa| + d+1)}
       {\prod_{i=1}^{d+1}\Gamma(\kappa_i +1)} \int_{T^d} f(x) g(x) W_\k(x) dx. 
\end{equation}
All three inner products are normalized so that $\la 1, 1\ra =1$. 

Let $\CV_n^d(U)$ denote the space of orthogonal polynomials of degree $n$ with respect to the weight function
$U$. It is known that 
$$ 
    \dim \CV_n^d(U) = \binom{n+d-1}{n},  \quad n =0,1,\ldots,
$$ 
where $n \le N$ if $U$ is either $\sH_\k$ or $\sK_\rho$. A polynomial $P \in \CV_n^d$ if $P$ is orthogonal
with respect to polynomials of degree less than $n$. The space $\CV_n^d$ can have many different bases. 
The reproducing kernel $P_n(\cdot,\cdot; U)$ of $\CV_n^d(U)$ is uniquely determined by the properties that 
$P_n(U; x, y) \in \CV_n^d(U)$ as a function of either $x$ or $y$ and 
$$
     \la  P_n(U; x, \cdot), Q \ra = Q (x), \qquad \forall Q \in \CV_n^d(U).
$$
In terms of orthonormal basis of $\{P_\nu: \nu \in \NN_0^d, |nu| =n\}$ of $\CV_n^d(U)$, the reproducing 
kernel can be written as 
$$
    P_n(x,y; U) = \sum_{|\nu| =n} P_\nu(x) P_{\nu}(y). 
$$
For $0 \le r < 1$, the Poisson kernel of the orthogonal polynomials is defined by
$$
    \Phi_r(x,y; U) : = \sum_{k=0}^N P_n(U; x,y) r^n, 
$$ 
where the sum is from zero to $\infty$ in the case of $U = W_\k$. When $U$ is either $\sH_{\k,N}$ or 
$\sK_{\rho,N}$, we denote the kernel by $\sP_n(U; \cdot, \cdot)$. 
 
The three families of orthogonal polynomials are closed related. In the case of one variable, the Hahn polynomials 
$\sQ(\cdot;a,b,N)$ can be defined with the Jacobi polynomials $P_n^{(a,b)}$ as a generating 
function. More precisely, it was proved in \cite{KM1} that 
\begin{equation}\label{eq:1dHahn}
  (1+t)^N \frac{P_m^{(a,b)}(\frac{1-t}{1+t})} {P_m^{(a,b)}(1)} =  
        \sum_{n=0}^N \binom{N}{n} \sQ_m(n; a,b, N) t^n, \qquad m = 0, 1,\ldots N. 
\end{equation}
This generating function approach was used to define a family of Hahn polynomials, denoted by
$\{\sH_\nu(\cdot; \k, N):  |\nu| \le N, \nu \in \NN_0^d\}$, in \cite{KM}. The generating function was
recognized as essentially a Jacobi polynomial on the simplex in \cite{X05}. The Krawtchouck 
polynomials are limit of Hahn polynomials when $\k = t(\rho, 1-|\rho|)$ and $t \to \infty$ \cite{IX07}.  

In the present paper, we explore the connection between the Jacobi polynomials and the Hahn
polynomials of several variables and derive a number of properties for both families of polynomials.  
It will be shown that the Hahn polynomials appear as connecting coefficients between several families 
of orthogonal polynomials in $\CV_n^d(W_\k)$, enhancing a result in \cite{X05}, and the results are 
used to derive tight polynomial frames for $L^2(T^d, W_\k)$ discussed in \cite{W}. The Jacobi 
polynomials on the simplex has been studied extensively, see \cite{DX}. In particular, a 
closed-form formula for the reproducing kernel $P_n(W_\k; \cdot,\cdot)$ was found in \cite{X98}, 
which plays a vital role for the harmonic analysis of the orthogonal expansions with respect to
$W_\k$. The generating function relation 
between the Jacobi polynomials and the Hahn polynomials extents to their reproducing kernels, 
which allows us to derive a closed-form formula for the kernel $\sP_n(\sH_{\k, N}; x,y)$ of the Hahn 
polynomials, from which a closed-form formula for $\sP_n(\sK_{\rho, N};x,y)$ of the Krawtchouck 
polynomials follows through a limit process. As an application, we shall prove that the Poisson 
kernels for the Hanh polynomials and the Krawtchouck polynomials are nonnegative whenever 
$x, y \in \NN_0^d$ and $|x|, |y| \le N$, which appear to be new even for $d =1$. 

After finishing the first draft of the paper, we discovered recent works \cite{GS11} and \cite{GS13} 
that studied the Jacobi and Hahn polynomials of several variables in the context of probability 
theory. In \cite{GS13}, the reproducing kernels for the Hahn polynomials are studied for the 
Lancaster problem of characterizing canonical correlation coefficients of the Dirichlet multinomial 
distribution, which is a normalized Hahn weight function, and the closed-form formula was found 
by an entirely different method. The Poisson kernels and the Krawtchouck polynomials were not 
studied in \cite{GS13}. While our starting point is the generating function, the starting point in 
\cite{GS11,GS13} is an integral representation of the Hahn polynomials in terms of the Jacobi 
polynomials.  

The Hahn polynomials and the Krawtchouck polynomials of several variables have been studied extensively 
in connection with applications in probability theory. We refer to \cite{GS11,GS13} for further results and references
on the Hahn polynomials and refer to \cite{DG} for a rich resource and references for the Krawtchouk polynomials.
  
The paper is organized as follows. In the next section, we introduce notations and basics for the Jacob,
Hahn and Krawtchouck polynomials. The generating function relations are used to derive results on the 
Hahn and the Jacobi polynomials in Section 3. The kernel functions for the Hahn polynomials
are discussed in Section 4. Finally, results on the Krawtchouck polynomials are established in Section 5.

\section{Orthogonal polynomials of several variables} 
\setcounter{equation}{0}

The notations of orthogonal polynomials can be overwhelming. We start with a subsection that explains 
the notations used in this paper. 

\subsection{Notations}

Throughout this paper, we reserve the greek letter $\k$ for the parameters in the Jacobi and the Hahn 
weight functions, and use greek letters other than $\k$ for multiindex. We follow standard practice for
multiindex notation. For $\nu \in \NN_0^d$ and $x \in \RR^d$, we let 
$$
          x^\nu  :=x_1^{\nu_1}\ldots x_d^{\nu_d}  
$$
and call $|\nu| = \nu_1+ \cdots +\nu_d$ the (total) degree of $x^\nu$. For $a \in \RR$ and $n \in \NN_0^d$, 
$$
(a)_n := a(a+1) \cdots (a+n-1)
$$ 
is the Pochhammer symbol. For $\nu, \mu \in \NN_0^d$ we let 
$$
  \nu ! := \nu_1! \cdots \nu_d! \quad\hbox{and}\quad  (\nu)_\mu := (\nu_1)_{\mu_1} \cdots (\nu_d)_{\mu_d},
$$
and we define $\nu \le \mu$ if $\nu_i \le  \mu_i$ for all $i$ and the inequality $\nu < \mu$ is defined similarly. We 
further define $\one = (1,\ldots,1)$, the vector whose components are all $1$. 

To describe our orthogonal polynomials, we will need the following notations: For $y=(y_1,\ldots, y_{d}) 
\in \RR^{d}$ and $1 \le j \le d$, we define 
\begin{equation}\label{xsupj}
    \yb_j := (y_1, \ldots, y_j) \quad \hbox{and}\quad \yb^j := (y_j, \ldots, y_d), 
\end{equation}
and also define $\yb_0 := \emptyset$ and $\yb^{d+1} :=  \emptyset$. It follows that $\yb_d = \yb^1 = y$, 
and 
$$
   |\yb_j| = y_1 + \cdots + y_j,   \quad |\yb^j| = y_j + \cdots + y_d, \quad\hbox{and}\quad
   |\yb_0| = |\yb^{d+1}| = 0.
$$
For $\k = (\k_1,\ldots, \k_{d+1})$, we defined $\k^j := (\k_j, \ldots, \k_{d+1})$ for $1 \le j \le d+1$ and adopt
this notation also for other greek letters. For $\nu \in \NN_0^d$ and $\k \in \RR^{d+1}$, we define
\begin{equation}\label{eq:aj}
   a_j:=a_j(\kappa,\nu):=|\kappa^{j+1}| + 2 |\nu^{j+1}| + d-j, \qquad 1 \le j \le d.
\end{equation} 
Notice that $a_{d} = \k_{d+1}$ since $|\nu^{d+1}| =0$ by definition. 

For the Hahn and Krwatchouck polynomials of several variables, it is often more convenient to consider them
as functions in homogeneous coordinates of $\RR^{d+1}$ or $\ZZ^{d+1}$. We define 
$$
     \RR_N^{d+1}:= \{ x: x \in \RR^d, |x| = N\}  \quad \hbox{and} \quad \ZZ_N^{d+1} := \{\a: \a \in \ZZ^d, |\a| = N\}.
$$
Using this convention and our multiindex notation, the Hahn weight function \eqref{eq:weightH} and the Krawtchouk weight function 
\eqref{eq:weightK} can be written as 
\begin{equation}\label{eq:HK-homo}
  \sH_{\k,N} (x) =  \frac{(\k + 1)_x}{x!} \quad \hbox{and}\quad  \sK_{\nu,N}(x) =  \frac{\brho^x}{x!},   \qquad 
    x \in \RR_N^{d+1}, 
\end{equation}
where $\brho = (\rho, \rho_{d+1})$ with $\rho \in (0,1)^d$ and $\rho_{d+1} = 1- |\rho|$. The inner product
\eqref{eq:ipdH} and \eqref{eq:ipdK} can be written as, with $\sU = \sH$ or $\sK$,  
\begin{equation}\label{eq:sum-homo}
     \la f, g\ra_{\sU_{\k,N}} = \sum_{\a \in \ZZ_N^{d+1}, |\a| = N} f(\a) g(\a) \sU_{\k,N}(\a).
\end{equation}

Finally, we will adopt the following conventions for this paper: 

\begin{enumerate}[\quad $\bullet$]
\item {\it Multiindex}: we use $\a, \b, \g$ for the multiinex in $\ZZ_N^{d+1}$ or $\NN_0^{d+1}$ and 
use $\nu$ and $\mu$ for the multiindex in $\NN_0^d$. 

\item {\it Orthogonal polynomials}: We use $\sH_\nu$ and $\sK_\nu$,  in sans serif font, for the 
Hahn and Krawtchouk polynomials, and use $P_\nu$ and $R_\nu$, in italic font, for the Jacobi polynomials. 
\end{enumerate}

\subsection{Orthogonal polynomials on the simplex}

For $x \in T^d$, we denote the homogeneous coordinates of $x$ by
\begin{equation}\label{eq:X}
   X = X(x):= (x_1,\ldots, x_d, x_{d+1}), \quad x_{d+1} = 1-x_1 -\cdots - x_d. 
\end{equation}
The normalization constant of the inner product $\la \cdot,\cdot \ra_{W_\k}$ 
\eqref{eq:innerW} of the Jacobi weight function $W_\k$ is determined by the beta integral on the simplex 
\begin{equation} \label{beta-integral}
 \int_{T^d} W_\k(x) dx = \int_{T^d} X^\k d x =  \frac{\prod_{i=1}^{d+1}\Gamma(\kappa_i +1)}{\Gamma(|\kappa| + d+1)}. 
\end{equation}
Let $\CV_n^d(W_\k)$ denote the space of orthogonal polynomials of degree $n$ with respect to $W_\k$. 
A polynomials $P \in \CV_n^d(W_\k)$ if $\la P, Q\ra_{W_\k} = 0$ for all polynomials $Q$ of degree at
most $n-1$.  A basis $\{P_\nu^n: |\nu| = n \}$ of $\CV_n^d(W_\k)$ is called mutually orthogonal if 
$\la P_\nu, P_\mu \ra_{W_\k} = 0$ whenever $\nu \ne \mu$ and $|\nu| = |\mu| =n$ and it is called 
orthonormal if, in addition, $\la P_\nu, P_\mu \ra = 1$. 

We give two bases of $\CV_n^d(W_\k)$. The first one is given in terms of the classical Jacobi polynomials
\begin{equation}\label{eq:jacobi}
   \frac{P_n^{(a,b)}(t)}{P_n^{(a,b)}(1)}
      =  {}_2 F_1 \left( \begin{matrix} -n, n+a+b+1\\ a+1 \end{matrix}; \frac{1-t}{2} \right).  
\end{equation}   

\begin{prop}
For $\k \in \RR^{d+1}$ with $\k_i > -1$, $\nu \in \NN_0^d$ and $x \in \RR^d$, define
\begin{equation}\label{eq:Pnu}
P_\nu(x) =P_\nu^\kappa (x) := \prod_{j=1}^d \left(\frac{1-|\xb_j|}{1-|\xb_{j-1}|}
 \right)^{|\nu^{j+1}|}  \frac{P_{\nu_j}^{(a_j,\kappa_j)}\left 
  (\frac{2x_j}{1-|\xb_{j-1}|} -1\right)}{P_{\nu_j}^{(a_j,\kappa_j)}(1)}, 
\end{equation}
where $a_j = a_j(\k, \nu)$ is defined in \eqref{eq:aj}. The polynomials in $\{P_\nu:|\nu|=n\}$ form a 
mutually orthogonal basis of $\CV_n^d(W_\k)$ with 
$A_\nu  = \langle P_\nu, P_\nu \rangle_{W_\kappa}$ given by  
\begin{align} \label{eq:norm1}
  A_\nu = A_\nu(\k) := 
   \frac{1}{(|\kappa|+d+1)_{2|\nu|}}
     \prod_{j=1}^d \frac{(\k_j+a_j+1)_{2 \nu_j} (\k_j +1)_{\nu_j} \nu_j!} {(\kappa_j+a_j+1)_{\nu_j}(a_j+1)_{\nu_j}}.
\end{align}
\end{prop}

The basis given in \eqref{eq:Pnu} is standard; see \cite[p. 47]{DX}. Notice, however, that our definition 
differs from the one given in \cite{DX} because of $P_{\nu_j}^{(a_j,\kappa_j)}(1)$ in the denominator. 

Our second family of polynomials in $\CV_{|\a|}^d(W_\k)$ consists of monic orthogonal polynomials.
For every $\alpha \in \NN_0^{d+1}$ and $x \in \RR^d$,  the monic orthogonal polynomial
$R_\a \in \CV_{|\a|}^d(W_\k)$ is defined uniquely by
$$
   R_\a(x) = R_\a^\k(x) := X^\a + q_\a(x), \qquad q \in \Pi_{|\a|-1}^d. 
$$
An explicit formula of $R_\a^\k$ in a Lauricella function of type $A$ was derived  in \cite{X05}. 

\begin{prop} \label{thm:monincR}
For $\a \in \NN_0^{d+1}$ and $x \in \RR^d$,  
\begin{align}\label{eq:Ra}
R_\a^\k(x)  & =  \frac{ (-1)^{|\a|} (\k+\mathbf{1})_\a}{(|\k|+d+|\a|)_{|a|}}
    F_A(|\k|+ d+|\a|, -\a; \k+\mathbf{1}; X) \\ 
    & =   \frac{ (-1)^{|\a|} (\k+\mathbf{1})_\a}{(|\k|+d+|\a|)_{|\a|}}
      \sum_{\g \le \a} 
      \frac{(-\a)_{\g} (|\k|+ d+|\a|)_{|\g|}}{( \k+\mathbf{1})_{\g}{}\g!} X^{\g}. \notag
\end{align}
Furthermore, $\{R_\a^\k: |\a| = n, \a_{d+1} = 0\}$ is a basis of $\CV_n^d(W_\k)$. 
\end{prop}

For each $\a \in \NN_0^{d+1}$ with $|\a| = n$, $R_\a$ is an element of $\CV_n^d(W_\k)$. The
number of such $R_\alpha$, however, is equal to $\# \{\a \in \NN_0^{d+1}: |\a| = n\} = \binom{n+d}{n}$,
which is significantly greater than $\dim \CV_n^d(W_\kappa) = \binom{n+d -1}{n}$. Thus, there are
substantial redundancy in the set $\{R_\alpha^\kappa:|\alpha|=n\}$ and the set is not a basis of 
$\CV_n^d(W_\kappa)$. The redundancy shows that we can expand $X^\a$ in terms of $\{R_\b: \b \le \a\}$ 
in many different ways. An explicit expansion that is a direct inverse of \eqref{eq:Ra} is given below,
which will be used in the next section. 

\begin{prop} \label{prop:Xa=Rb}
For each $\alpha \in \NN_0^{d+1}$ and $x \in \RR^d$, 
\begin{align}\label{eq:Xa=Rb}
  X^\a = (\k + \one)_\a   \sum_{\b \le \a}\frac{(-1)^{|\b|} (-\a)_{\b} (|\k|+d+1)_{2|\b|}}
       {\b! (\k + \one)_\b (|\k|+d+1)_{|\a|+|\b|}} R_\b^\k(x).
\end{align}
\end{prop}

\begin{proof}
Let denote the right hand side by RHS. By \eqref{eq:Ra}, it follows that 
$$
   \textrm{RHS} = \sum_{\b \le \a}a_{\a,\b} R_\b^\k(x) =  \sum_{\b \le \a}a_{\a,\b} \sum_{\g\le \b} b_{\b,\g} X^\g,
$$
where the values of $a_{\a,\b}$ and $b_{\b,\g}$ can be read out from \eqref{eq:Xa=Rb} and \eqref{eq:Ra},
respectively. Changing the order of summations and then the summation index by $\g \to \a-\g$ and $\b \to \a - \b$, 
we obtain 
$$
 \textrm{RHS}  =  \sum_{\g \le \a}\sum_{\g\le \b \le \a} a_{\a,\b}  b_{\b,\g} X^\g
 =  \sum_{\g \le \a} X^{\a -\g} \sum_{\b \le \g} a_{\a,\a-\b} b_{\a- \b,\a -\g}.
$$
Thus, we need to prove that the sum $\sum_{\b\le \g}$ is zero whenever $\g \ne 0$, and is $1$ when $\g =0$.
The case $\g = 0$ is trivial. We now consider the case $\g \ne 0$. Using 
$$
(a)_{n-m} = \frac{(-1)^m (a)_n}{(1-a-n)_m} \quad\hbox{and} \quad (-\a+\b)_{\a-\g} = (-1)^{|\a|+|\g|} \frac{\a! (-\a)_\b}
   {\g! (-\g)_\b}
$$ 
to rewrite $a_{\a,\a-\b} b_{\a- \b,\a -\g}$, it is not difficult to see that 
\begin{align*}
  \sum_{\b \le \g} a_{\a,\a-\b} b_{\a- \b,\a -\g} = \frac{(-\a)_\g (-\k-\a)_\g}{\g!}
    \sum_{\b \le \g}  \frac{(-\g)_\b}{\b!} \frac{ (- \rho + 1)_{2|\b|} (- \rho)_{|\b|}}
       {(- \rho + |\g|+1)_{|\b|}(- \rho)_{2|\b|}},
\end{align*}
where $\rho = |\k| + d +2|\a|$. It is easy to verify that 
$$
   \frac{ (- \rho + 1)_{2|\b|} (- \rho)_{|\b|}} {(- \rho)_{2|\b|}} = 2 (-\rho + 1)_{|\b|} - (-\rho)_{|\b|}
     \quad\hbox{and} \quad  \sum_{|\b| =n} \frac{(-\g)_\b}{\b!} = \frac{(-|\g|)_n}{n!},
$$
which allow us to write the last sum over $\beta \le \g$ as 
\begin{align*}
  \sum_{\b \le \g}   = 2 {}_2 F_1 \left(\begin{matrix} -|\g|, -\rho+1\\ -\rho + |\g|+1 \end{matrix}; 1\right) 
     -  {}_2 F_1 \left(\begin{matrix} -|\g|, -\rho\\   -\rho + |\g|+1 \end{matrix}; 1\right).  
\end{align*}
Applying the Chu-Vandermonde identity, we then conclude that 
$$
   \sum_{\b \le \g} = \frac{2 (|\g|)_{|\g|} -  (|\g|+1)_{|\g|}}{(-\rho+|\g|+1)_{|\g|}} =0, \qquad |\g| \ne 0,
$$ 
which shows that RHS = $0$ whenever $\g \ne 0$ and completes the proof. 
\end{proof}

Let $P_n(W_\k; \cdot,\cdot)$ be the reproducing kernel of $\CV_n^d(W_\k)$. In terms of the mutually 
orthogonal basis of $\{P_\nu:|\nu|=n\}$, this kernel can be written as
\begin{equation}\label{eq:Preprod}
  P_n(W_\k;x,y) = \sum_{|\nu| = n} \frac{P_\nu^\k(x) P_\nu^\k(y)}{A_\nu(\k)}, \qquad x, y \in \RR^d.
\end{equation}
This kernel satisfies a closed-form formula if $\k_i \ge -1/2$ for $1 \le i \le d+1$ (\cite{X98}).  

\begin{thm}
Assume $\k_i \ge -1/2$ for $i=1,\ldots, d+1$. For $x, y\in T^d$, 
\begin{align}\label{eq:Pclosed-form}
 & P_n(W_\k; x,y)  = \frac{(|\k|+d+1)_n (|\k|+d+2n)}{(\frac12)_n (|\k|+d+n)} \\
& \qquad \times c_\k \int_{[-1,1]^{d+1}} P_n^{(|\k|+d -\f12, -\f12)} (2 z(t; x,y)^2 -1) \prod_{i=1}^{d+1} (1-t_i^2)^{\k_i-\f12} dt,
   \notag
\end{align}
where $x_{d+1} = 1-|x|$,  $y_{d+1} = 1-|y|$, and
$$
   z(t,x,y) = \sqrt{x_1 y_1} t_1+ \cdots + \sqrt{x_{d+1} y_{d+1}} t_{d+1},
$$
and $c_\k$ is the normalization constant so that $ P_0(W_\k; x,y) =1$. 
\end{thm}

If any one of $\k_i = -1/2$, the formula \eqref{eq:Pclosed-form} holds under an appropriate limit. This closed-form is responsible for many recent work on the harmonic analysis of orthogonal expansions on the simplex.  

\subsection{Hahn polynomials of several variables}
For $a, b > -1$, the classical Hahn polynomial $Q_n(x; a,b,N)$ is a ${}_3F_2$ hypergeometric function given by
\begin{equation}\label{eq:HahnQ}
  \sQ_n(x; a, b, N) := {}_3 F_2 \left( \begin{matrix} -n, n+a + b+1, -x\\
       a+1, -N \end{matrix}; 1\right), \qquad n = 0, 1, \ldots, N.    
\end{equation}
These are discrete orthogonal polynomials defined on the set $\{0,1,\ldots, N\}$. It is known that 
they satisfy (see, for example, \cite[p.\ 346]{AAR}) 
\begin{equation}\label{QnQn}
  \sQ_n(x; a,b, N) = (-1)^n \frac{(b+1)_n}{(a+1)_n} \sQ_n(N-x;b,a,N).
\end{equation}

A family of Hahn polynomials of several variables were defined and studied in \cite{KM} through 
a generating function that extends \eqref{eq:1dHahn}. The generating function was recognized as
$$
 P_{\nu, N}(y) := |y|^{N-|\nu|}\prod_{j=1}^d |\yb^j|^{\nu_j} \frac{ P_{\nu_j}^{(a_j,\kappa_j)}\left( 2\f{y_j}{|\yb^j|} -1\right)}
  {P_{\nu_j}^{(a_j,\kappa_j)}(1)}, \qquad y \in \RR^{d+1}, 
$$
in \cite{X05}, where $a_j = a_j(\k,\nu)$ are defined in \eqref{eq:aj}, which can be written in terms of the Jacobi 
polynomial $P_\nu(x)$ on the simplex, defined in \eqref{eq:Pnu}, if we replace $y$ by $y /|y|$ and using 
$|\yb^j| = |y| - |\yb_{j-1}|$.  

\begin{prop} \label{def:Hahn}
Let $\k\in\RR^{d+1}$ with $\kappa_i>-1$ and $N\in\NN$. For $\nu\in\NN_0^d$, $|\nu|\le N$, define
the Hahn polynomials $\sH_\nu(\a; \k ,N)$ for $\a \in \ZZ_N^{d+1}$ by 
\begin{equation}\label{Hahngenfunc}
 P_{\nu,N}(y) = |y|^N P_\nu\Big ( \f {y'} {|y|} \Big)
   = \sum_{|\alpha| = N} \frac{N!}{\alpha!}\sH_\nu(\alpha; \kappa,N)y^\alpha,
\end{equation}
where $y = (y', y_{d+1}) \in \RR^{d+1}$. 
\end{prop}

If $d=1$ then $a_1(\k,\nu) = \k_2$ and $\nu$ is a scalar; setting $y = (1,t)$ and taking the
summation over $\a = (N- n, n)$ for $n = 0,1,\ldots, N$,  \eqref{Hahngenfunc} becomes 
$$ 
 (1+t)^N \frac{P_\nu^{(\k_2,\k_1)}(\frac{1-t}{1+t})}
   {P_{\nu}^{(\kappa_2,\kappa_1)}(1)} = \sum_{j=0}^N \binom{N}{j} \sH_n( (N-j,j); \kappa_1,\kappa_2,N)t^j. 
$$
Comparing with \eqref{eq:1dHahn} shows that $\sH_\nu ((N-j,j); \k_1,\k_2,N)$ is the Hahn polynomial
$\sQ_\nu(j;\k_2,\k_1,N)$ of one variable. Using \eqref{QnQn}, we obtain that, for $x \in \ZZ_N^2$,
\begin{equation}\label{eq:phi_d=1}
    \sH_\nu (x; \k_1,\k_2,N)  =   \frac{(-1)^\nu(\k_1+1)_\nu}{(\k_2+1)_\nu} \sQ_\nu (x_1; \k_1,\k_2,N), \quad 0 \le j \le N,
\end{equation}

The explicit formula for $\sH_\nu(\cdot; \kappa,N)$ and its norm were stated in \cite{KM} by 
inductive formulas and they were explicitly given in \cite{X05}  (see \cite{IX07} for a direct proof).  

\begin{prop} \label{prop:Hahn1}
For $x \in \ZZ_N^{d+1}$ and $\nu \in \NN_0^d$, $|\nu| \le N$, 
\begin{align}\label{eq:Hn-prod}
\sH_\nu(x;\kappa, N) =& \frac{(-1)^{|\nu|}}{(-N)_{|\nu|}} 
\prod_{j=1}^d  \frac{(\kappa_j+1)_{\nu_j}}{(a_j+1)_{\nu_j}}(-N+|\xb_{j-1}|+|\nu^{j+1}|)_{\nu_j}  \\ 
    & \times  
    \sQ_{\nu_j}(x_j; \kappa_j, a_j, -N+|\xb_{j-1}|-|\nu^{j+1}|), \notag
\end{align}
where $a_j = a_j(\k,\nu)$ as in \eqref{eq:aj}. 
The polynomials in $\{\sH_\nu(x; \kappa,N): |\nu| = n\}$ form a mutually orthogonal basis of $\CV_n^d(\sH_{\k,N})$ 
and $\sB_\nu := \la \sH_\nu(\cdot; \kappa,N), \ \sH_\nu(\cdot; \kappa,N) \ra_{\sH_{\k,N}}$ 
is given by, setting $\l_\k : = |\k|+d+1$, 
\begin{align*} 
\sB_\nu(\k, N)  :=\frac{(-1)^{|\nu|}(\l_k)_{N+|\nu|}}
   {(-N)_{|\nu|} (\l_k)_N (\l_k)_{2|\nu|}} 
   \prod_{j=1}^d \frac{(\k_j+a_j+1)_{2\nu_j}  (\k_j+1)_{\nu_j} \nu_j! }
    { (\kappa_j+a_j + 1)_{\nu_j} (a_j+1)_{\nu_j}}.
\end{align*}
\end{prop}

A useful observation is that $\sB_\nu(\k,N)$ and $A_\nu(\k) = \la P_\nu, P_\nu\ra_{W_\k}$ differ by a 
constant that depends on $|\nu|$ but not elements in $\nu$.

\begin{cor}
For $\nu \in \NN_0^d$,
\begin{equation} \label{eq:B-A}
  \sB_\nu(\k,N) = \frac{(-1)^{|\nu|}(|\k|+d+1)_{N+|\nu|}} {(-N)_{|\nu|}\, (|\k|+d+1)_N} A_{\nu}(\k). 
\end{equation}
\end{cor}

The reproducing kernel $\sP_n(\sH_{\k, N};\cdot,\cdot)$ is independent of the bases of $\CV_n^d(\sH_\k)$. In 
terms of the mutual orthogonal basis in Proposition \ref{prop:Hahn1}, it can be written as 
\begin{equation}\label{Hahn-reprod}
  \sP_n(\sH_{\k, N};x, y) = \sum_{|\nu| = n} \frac{\sH_\nu(x; \k, N)\sH_\nu(y; \k, N)}{\sB_\nu(\k, N)}. 
\end{equation}

\subsection{Krawtchouck polynomials of several variables}
For $0 < p < 1$, the classical Krawtchouck polynomial $\sK_n(x; p, N)$ of one variable is
defined by 
\begin{equation}\label{eq:KrawKd=1}
  \sK_n(x; p, N) := {}_2 F_1 \left( \begin{matrix} -n, -x\\
       -N \end{matrix}; \frac{1}{p}\right), \qquad n = 0, 1, \ldots, N.    
\end{equation}
These are discrete polynomials on $\{0,1,\ldots,N\}$. They satisfy the relation
\begin{equation}\label{KnKn}
   \sK_n(N-x, p, N) = \frac{(-1)^n p^n}{(1-p)^n} \sK_n(x, 1-p, N). 
\end{equation}
For $\nu \in \NN_0^d$, a family of the Krawtchouck polynomials, denoted by $\sK_\nu(\cdot;\rho,N)$, 
of $d$ variables can be given in terms of Krawtchouck polynomials of one variable \cite{IX07, Tr2}. 
For $\rho \in \NN_0^d$, we use the notation \eqref{xsupj}, which implies that $|\brho_j| = \rho_1+\cdots + \rho_j$
for $j =1,2,\ldots, d$. 

\begin{prop} \label{prop:Kraw}
Let $\rho \in \RR^d$ with $0 < \rho_i < 1$ and $|\rho| < 1$. 
For $\nu\in \NN_0^d$, $|\nu|\le N$, and $x \in \RR^d$, define
\begin{align}\label{eq:KrawK}
  \sK_\nu(x; \rho, N) := & \frac{(-1)^{|\nu|} }{(-N)_{|\nu|}} 
    \prod_{j=1}^d  \frac{\rho_j^{\nu_j}} {(1- |\brho_j|)^{\nu_j}}
       (-N+|\xb_{j-1}|+|\nu^{j+1}|)_{\nu_j} \\
  &   \times \sK_{\nu_j}\left(x_j; \frac{\rho_j}{1-|\brho_{j-1}|}, 
         N-|\xb_{j-1}|- |\nu^{j+1}|\right). \notag
\end{align}
The polynomials in $\{\sK_\nu(\cdot;\rho, N): |\nu| = n\}$ form a mutually orthogonal basis of $\CV_n^d(\sK_{\rho,N})$ 
and $\sC_\nu (\rho, N): = \la \sK_\nu(\cdot;\rho, N), \sK_\nu(\cdot;\rho, N) \ra_{\sK_{\rho, N}}$ is given by 
$$
\sC_\nu (\rho, N): =  \frac{(-1)^{|\nu|}}{(-N)_{|\nu|}N!} \prod_{j=1}^d \frac{\nu_j! \rho_j^{\nu_j} } 
    {(1-|\brho_j|)^{\nu_j-\nu_{j+1}} }.
$$
\end{prop}
 
The Krawtchouck polynomials in \eqref{eq:KrawK} are limits of the Hahn polynomials in 
\eqref{eq:Hn-prod}. More precisely, setting $\k = t (\rho, 1-|\rho|)$, we have (\cite{IX07})
\begin{align} \label{Hahn-Kraw}
  \lim_{t \to \infty} \sH_\nu(x; t(\rho, 1-|\rho|), N)   = \sK_\nu(x; \rho, N).
\end{align}

The reproducing kernel $\sP_\nu(\sK_{\rho,N};\cdot,\cdot)$ is again independent of the bases in 
$\CV_n^d(\sK_{\rho,N})$. In terms of the basis in \eqref{eq:KrawK}, it can be written as 
\begin{equation}\label{repod-Kraw}
 \sP_\nu(\sK_{\rho,N};x,y) = \sum_{|\nu|=n} \frac{ \sK_\nu(x; \rho, N) \sK_\nu(x; \rho, N)}{\sC_n (\rho, N)}.
\end{equation}

\section{Jacobi polynomials and Hahn polynomials}
\setcounter{equation}{0} 

Our first result is to use the generating function to derive connecting coefficients between the mutually
orthogonal polynomials $P_\nu^\k$ in \eqref{eq:Pnu} and the monic orthogonal polynomials 
$R_\a^\k$ in \eqref{eq:Ra} of $\CV_n^d(W_\k)$.

\begin{thm} \label{thm:Pnu-Ra}
For $\nu \in \NN_0^d$ with $|\nu|= n$ and $\a\in \ZZ_n^{d+1}$, 
\begin{equation} \label{eq:Pnu=Ra}
   P_\nu^\k (x) = \sum_{|\a| =n} \frac{n!}{\a!} \sH_\nu(\a; \k,n) R_\a^\k(x),
\end{equation}
and, conversely, 
\begin{equation} \label{eq:Ra=Pnu}
   R_\a^\k(x) =  \frac{(\k+\one)_\a}{(|\k|+d+1)_n} \sum_{|\nu| =n} \frac{ \sH_\nu(\a; \k,n)}{\sB_\nu(\k,n)} P_\nu^\k(x).
\end{equation}
\end{thm}

\begin{proof}
By definition, $R_\a^\k(x) = X^\a + q_\a$ with $q_\a$ being a polynomial of degree at most $n-1$. Hence,
setting $y = X$, so that $|y| =1$, in \eqref{Hahngenfunc} shows
$$
     P_\nu^\k (x) = \sum_{|\a| =n} \frac{n!}{\a!} \sH_\nu(\a; \k, n) R_\a^\k (x) + q(x),
$$
where the degree of $q$ is at most $n-1$. However, both $P_\nu^\k$ and $R_\a^\k$ are orthogonal polynomial
of degree $n$, so that $q(x) =0$. This proves \eqref{eq:Pnu=Ra}.  

Since $\{P_\mu^\k: |\mu| = n\}$ is a basis of $\CV_n^d(W_\k)$, there are unique constants $c_{\a,\mu}$ such 
that $R_\a^\k(x) = \sum_{|\mu| = n} c_{\a,\mu} P_\mu^\k(x)$. Hence, by \eqref{eq:Pnu=Ra},
we must have
$$
 P_\nu^\k(x) = \sum_{|\a| =n} \frac{n!}{\a!} \sH_\nu(\a; \k,n) \sum_{|\mu| = n} c_{\a,\mu} P_\mu^\k(x)
       = \sum_{|\mu| = n} \sum_{|\a| =n} \frac{n!}{\a!} \sH_\nu(\a; \k,n) c_{\a,\mu} P_\mu^\k(x). 
$$
Since $P_\nu^\k$ are mutually orthogonal, we must have 
$$
\sum_{|\a| =n} \frac{n!}{\a!} \sH_\nu(\a; \k,n) c_{\a,\mu} = \delta_{\nu,\mu}.
$$
The orthogonality of $\sH_\nu(\cdot;\k,n)$ gives one solution of $c_{\a,\mu}$. The uniqueness of 
$c_{\a,\mu}$ shows that it is the only solution, which proves, recall that our inner product is normalized, 
the identity \eqref{eq:Ra=Pnu}.
\end{proof}
 
Since \eqref{eq:Ra=Pnu} gives an expansion of $R_\a^\k$ in terms of mutually orthogonal polynomials, 
we obtain from \eqref{eq:B-A} with $N = n$ the following corollary.  

\begin{cor}
For $\nu \in \NN_0^d$ and $\a \in \NN_0^{d+1}$ with $ |\a| = |\nu| =n$, 
\begin{equation} \label{eq:Pnu,Ra}
       \la P_\nu^\k, R_\a^\k \ra_{W_\k}   = \frac{n!(\k+\one)_\a } {(|\k|+d+1)_{2n}} \sH_{\nu}(\a; \k, n). 
\end{equation}
\end{cor}
  
As another corollary, we see that the inner product $\la R_\a^\k, R_\beta^\k\ra_{W_\k}$, which has not
been computed before, is given in terms of the reproducing kernel of $\CV_n^d(\sH_{\k,n})$. 

\begin{cor}
For $\a,\b \in \NN_0^{d+1}$ and $|\a| = |\b| = n$, 
\begin{equation}\label{Pn=Rab}
\la R_\a^\k,R_\b^\k \ra_{W_\k} =   \frac{n! (\k+\one)_\a (\k+\one)_\b} {(|\k|+d+1)_n(|\k|+d+1)_{2n}} 
      \sP_n(\sH_{\k, n}; \a,\b).         
\end{equation}
\end{cor}

\begin{proof}
By \eqref{eq:Ra=Pnu} and the orthogonality of $P_\nu^\k$,  
\begin{align*}  
 \la R_\a^\k,R_\b^\k \ra_{W_\k} = \frac{(\k+\one)_\a (\k+\one)_\b}{[(|\k|+d+1)_n]^2} \sum_{|\nu| = n} \sum_{|\mu| =n}
    \frac{\sH_\nu(\a; \k,n)\sH_\mu(\b; \k,n)}{\sB_\nu(\k,n)\sB_\mu(\k,n)} A_\nu(\k) \delta_{\nu, \mu}. 
\end{align*}
By \eqref{eq:B-A}, $A_\nu(\k)/ \sB_\nu(\k,N)$ can be taken out of the summation, so that the stated result follows 
from \eqref{eq:B-A} and \eqref{Hahn-reprod}.
\end{proof}

\begin{thm}
For $\nu \in \NN_0^d$ with $|\nu| = n$ and $\a \in \ZZ_N^{d+1}$, 
\begin{equation} \label{eq:Pnu,Xa}
\sH_\nu(\a; \k, N) =  \frac{(-1)^n (|\k|+d+1)_{N + n}} { (- N)_n  (\k+\one)_\a} \la P_\nu^\k, X^\a \ra_{W_\k}. 
\end{equation}
\end{thm}

\begin{proof}
By the orthogonality of $P_\nu^\k$ and \eqref{eq:Xa=Rb}, 
\begin{align*}
    \la P_\nu^\k, X^\a \ra_{W_\k}   = (\k + \one)_\a 
     \sum_{|\b| = n}\frac{(-1)^{|\b|} (-\a)_{\b} (|\k|+d+1)_{2 n}}
           {\b! (\k + \one)_\b (|\k|+d+1)_{|\a| + n}} \la P_\mu^\k, R_\b^\k \ra_{W_\k}   
\end{align*}
Hence, it follows from \eqref{eq:Pnu,Ra} that 
\begin{equation} \label{eq:Pnu,Xa2}
  \la P_\nu^\k, X^\a \ra_{W_\k}  = \frac{(\k + \one)_\a n!} {(|\k|+d+1)_{N+n}} \sum_{|\b| =n}\frac{(-1)^{|\b|} (-\a)_{\b}}
           {\b!} \sH_\nu(\b; \k, n),
\end{equation}
which shows, in particular, that $\la P_\nu^\k, X^\a \ra_{W_\k}/(\k + \one)_\a$ is a polynomial of
degree at most $n$ in $\alpha$. On the other hand, by \eqref{Hahngenfunc},  
$$
  A_\nu(\k) \delta_{\nu,\mu} = \la P_\nu^\k, P_\mu^\k \ra_{W_\k} = 
    \sum_{|\a| = N} \frac{N!}{\a!} \sH_\mu(\a; \k,N) \la  P_\nu^\k, X^\a \ra_{W_\k},
$$
so that, by the orthogonality of the Hahn polynomials, $\la P_\nu^\k, X^\a \ra_{W_\k}/(\k + \one)_\a$
is uniquely determined and the above equation gives 
\begin{equation} \label{eq:Pnu,Xa3}
  \la P_\nu^\k, X^\a \ra_{W_\k} =  \frac{(\k+\one)_\a A_\nu(\k)}{(|\k|+d+1)_N \sB_\nu(\k,N)} \sH_\nu (\a; \k,N),
\end{equation}
from which \eqref{eq:Pnu,Xa} follows from \eqref{eq:B-A}. This completes the proof. 
\end{proof}

A different proof of \eqref{eq:Pnu,Xa} is given in \cite[Proposition 5.2]{GS11}. This identity is taken as 
the starting point of the approach in \cite{GS13}. 

Since the expansion of $X^\a$ in terms of orthogonal basis $\{P_\nu^\k: |\nu| \le |\a|\}$ has 
$\la P_\nu^\k, X^\a \ra_{W_\k} / A_\nu(\k)$ as coefficients, \eqref{eq:Pnu,Xa3} yields the 
following inverse of the generating function. 

\begin{cor}
For $\a \in \ZZ_N^{d+1}$ and $x\in \RR^d$, 
$$
  X^\a = \frac{(\k+\one)_\a}{ (|\k|+d+1)_N} \sum_{|\nu| \le N}
              \frac{\sH_\nu(\a; \k, N)}{ \sB_\nu(\k, N)} P_\nu^\k(x).
$$
\end{cor}

Comparing \eqref{eq:Pnu,Xa} with \eqref{eq:Pnu,Xa2} yields the following corollary: 

\begin{cor} \label{cor:Hahn-Hahn}
For $\nu \in \NN_0^d$ with $|\nu| =n$ and $x \in \ZZ_N^{d+1}$, 
\begin{equation} \label{eq:phi-monic}
   \sH_\nu(x; \k, N) = \frac{n!}{(-N)_n} \sum_{|\a|=n} \frac{\sH_\nu(\a; \k, n)}{\a!}   (-x)_\a. 
\end{equation}
\end{cor} 

The identity \eqref{eq:phi-monic} is interesting since it gives the explicit expansion of $\sH_\nu(x; \k , N)$ in terms 
of the shifted monomials $\{(-x)_\a: \a \in \ZZ_n^{d+1}\}$. 

For any $\a \in \NN_0^{d+1}$ and $x \in \RR^d$, $X^\a$ is a polynomial of degree $|\a|$. We denote 
its orthogonal projection on $\CV_n^d(W_\k)$ by $R_{\a,n}^\k$.  If $|a| < n$, then $R_{\a,n}^\k =0$. If
$|\a|=n$, then $R_{\a,n}^\k = R_\a^\k$. For $|\a| = N \ge n$, by \eqref{eq:Xa=Rb}, we see that
\begin{align} \label{eq:Ran}
R_{\a,n}^\k(x) =  \frac{(\k + \one)_\a(|\k|+d+1)_{2n} } { (|\k|+d+1)_{|\a| +n}} 
        \sum_{|\b| =n} \frac{(-1)^{|\b|} (-\a)_{\b}} {\b! (\k + \one)_\b} R_\b^\k(x).
\end{align}
Substituting the expansion of $R_\b^\k$ in \eqref{eq:Ra} in the above equation, we can deduce an 
expansion of $R_{\a,n}^\k(x)$ in the power of $X$ which has appeared in \cite{W}. 

\begin{prop}
For $\a \in \ZZ_N^{d+1}$ and $\nu \in \NN_0^d$ with $|\nu| = n$, 
\begin{equation}\label{Ran,Pnu}
   \la R_{\a,n}^\k, P_\nu^\k \ra_{W_\k} = (\k +1)_\a \frac{ (-1)^n (-N)_n} {(|\k|+d+1)_{N+n}}\sH_\nu(\a; \k, N). 
\end{equation}
\end{prop}

\begin{proof}
By \eqref{eq:Ran} and \eqref{eq:Pnu,Ra}, 
\begin{align*}
 \la R_{\a,n}^\k, P_\nu^\k \ra_{W_\k} =  \frac{(\k +1)_\a n!} {(|\k|+d+1)_{N+n} } 
          \sum_{|\beta| = n} \frac{(-1)^{|\b|} (-\a)_\b}{\b!} \sH_\nu(\a; \k, n).
\end{align*}
The stated result then follows from \eqref{eq:phi-monic}. 
\end{proof}
 
\begin{cor}
For $\a \in \ZZ_N^{d+1}$, $\nu \in \NN_0^d$ with $|\nu| = n$ and $x \in \RR^d$, 
\begin{equation}\label{Ran =Pnu}
R_{\a,n}^\k(x) =  \frac{(\k +1)_\a} {(|\k|+d+1)_{N+n} } \sum_{|\nu| = n} 
      \frac{\sH_\nu(\a; \k, N)}{\sB_\nu(\k,N)} P_\nu^\k(x),
\end{equation}
and, conversely, 
\begin{equation}\label{Pnu=Ran}
    P_\nu^\k(x) =  \sum_{|\a| = N} \frac{N!}{\a!} \sH_\nu(\a; \k, N) R_{\a,n}^\k(x).  
\end{equation}
\end{cor}

\begin{proof}
Since $\la R_{\a,n}^\k, P_\nu^\k \ra_{W_\k}/A_\nu(\k)$ is the coefficient of the orthogonal expansion 
of $R_{\a,n}^\k(x)$, the identity \eqref{Ran =Pnu} follows from \eqref{Ran,Pnu} and \eqref{eq:B-A}.
The right hand side of \eqref{Pnu=Ran}, when $R_{\a,n}^\k(x)$ is replaced by \eqref{Ran =Pnu},
is equal to $P_\nu^\k(x)$ by the orthogonality 
of $\sH_\nu(\cdot; \k, N)$.
\end{proof}

Recall that $P_n(W_\k; \cdot,\cdot)$ denotes the reproducing kernel of $\CV_n^d(W_\k)$. 

\begin{thm}
For $\a\in \ZZ_N^{d+1}$, 
\begin{equation} \label{eq:frame}
  P_n(W_\k; x,y) = \frac{(|\k|+d+1)_{N+n} N!} {(-1)^n (-N)_n}
       \sum_{|\a| = N} \frac{1}{(\k+1)_\a \a!}  R_{\a,n}^\k(x) R_{\a,n}^\k(y).
\end{equation}
\end{thm}

\begin{proof}
Another way to write \eqref{Pnu=Ran}, by \eqref{Ran,Pnu} if necessary, is that 
$$
  \frac{(|\k|+d+1)_{N+n} N!} {(-1)^n (-N)_n}   \sum_{|\a| = N} \frac{1}{(\k+1)_\a \a!}
        \la R_{\a,n}^\k, P_\nu^\k \ra_{W_\k} R_{\a,n}^\k(x) = P_\nu^\k(x), 
$$
which shows that the kernel function defined by the right hand side of \eqref{eq:frame} reproduces 
$P_\nu^\k$. Since the kernel is clearly an element of $\CV_n^d(W_\k)$ as a function of either $x$ or $y$ 
and $\{P_\nu^\k: |\k| =n\}$ is a basis of $\CV_n^d(W_\k)$, it must be the reproducing kernel of $\CV_n^d(W_\k)$. 
\end{proof}

The cardinality of the set $\{R_{\a,n}^\k: |\a| = N\}$ is $\binom{N+d}{N}$, which is much larger than 
$\binom{n+d-1}{n}$, the dimension of $\CV_n^d(W_\k)$. The identity \eqref{eq:frame} implies that 
the system $\{R_{\a,n}^\k: |\a| = N\}$ is a tight frame of $\CV_n^d(W_\k)$ and, consequently, that 
$\{R_{\a,n}^\k: |\a| = N, n=0,1,\ldots\}$ is a tight frame of $L^2(T^d, W_\k)$, which was studied in 
\cite{W}. 
 
There is another way of computing the inner product of $P_\nu^\k$ and $X^\a$, which leads to 
the following lemma. 
  
\begin{lem}
For $\a \in \ZZ_M^{d+1}$,  
\begin{align} \label{eq:Psi-n}
  \sum_{|\b| = N} \frac{(\k + \one)_{\a+\b}}{\b!(\k+\one)_\a} \sH_\nu(\b; \k, N)  
  =  \frac{ (-M)_n (|\k|+d+1)_{N+M}}{(-1)^n N! (|\k|+d+1)_{n+M} } \sH_\nu(\a; \k, M), 
\end{align}
\end{lem} 
 
\begin{proof}
Using \eqref{Hahngenfunc} and the beta integral \eqref{beta-integral}, we obtain 
\begin{align*}
  \la P_\nu^\k, X^\a \ra_{W_\k} & = \sum_{|\b| = N} \frac{N!}{\b!} \sH_\nu(\b; \k, N)\la X^\b, X^\a\ra_{W_\k}  \\
    &  =  \sum_{|\b| = N} \frac{N!}{\b!} \sH_\nu(\b; \k, N) \frac{(\k + \one)_{\a+\b}} {(|\k| + d +1)_{|\a|+|\b|}}.
\end{align*}
Comparing this with \eqref{eq:Pnu,Xa} gives \eqref{eq:Psi-n}.  
\end{proof}

The identity \eqref{eq:Psi-n} deserves a second look. For $\k \in \RR^{d+1}$ with $\k_i > -1$ and 
$x,y \in \NN_0^{d+1}$, we define 
$$
    E_\k(x,y) : = \frac{(\k + \one)_{x+y}}{(\k + \one)_{x}(\k + \one)_{y}}.  
$$
For fixed $x$, $E_\k(x, \cdot)$ is a polynomials of degree $|x|$. The identity \eqref{eq:Psi-n} can
then be rewritten as  
\begin{equation} \label{eq:Psi-n2}
   \la  E_\k(x, \cdot), \sH_\nu(\cdot; \k, N) \ra_{H_{\k,N}} = 
       \frac{ (-M)_n (|\k|+d+1)_{N+M}}{(-1)^n (|\k|+d+1)_N (|\k|+d+1)_{n+M} }  \sH_\nu(\a; \k, M). 
\end{equation}
In particular, if $M =N$, then it shows that a constant multiple of $E_\k(\cdot,\cdot)$ reproduces 
the Hahn polynomials in $\CV_n^d(\sH_{\k,N})$. The kernel $E_\k(\cdot,\cdot)$ can be expressed 
in terms of the reproducing kernels of $\sP_n(\sH_{\k,N};\cdot,\cdot)$ as follows. 

\begin{prop}
For $\k \in \RR^{d+1}$ with $\k_i > -1$ and $x, y \in \ZZ_N^{d+1}$, 
\begin{equation}\label{eq:Ek2}
     E_\k (x,y) =  \frac{(|\k|+d+1)_{2N}}{(|k|+d+1)_N} 
         \sum_{n=0}^N \frac{(-1)^n (-N)_n}{(|k|+d+1)_{N+n}}  \sP_n(\sH_{\k,N}; x, y). 
\end{equation}
\end{prop}

\begin{proof}
Since $E_\k (x,\cdot)$ is a polynomial of degree $N$, it has an orthogonal expansion 
$E(x,y) = \sum_{|\nu| \le N} b_{\nu}(x) \sH_{\nu}(y; \k, N)$ for some coefficients $b_{\nu}(x)$ independent
of $y$. For fixed $x$, 
$$
  b_\nu(x) = \la E(x,\cdot), \sH_\nu(\cdot; \k, N)\ra_{\sH_{\k,N}}/ \sB_\nu(\k,N),
$$
from which \eqref{eq:Ek2} follows from  \eqref{eq:Psi-n2} and \eqref{Hahn-reprod}. 
\end{proof} 

The identity \eqref{eq:Psi-n2} appears to be nontrivial even in the case $d =1$ and $N =n$, which we 
examine below. If $d =1$, then $\nu$ is a scalar and we set $n = \nu$ and $x = (x_1,x_2)$ with $x_1+x_2 = n$. 
Since $\sQ_n(x_1;\k_1,\k_2,n)$ becomes a ${}_2F_1$ and can be summed up by the Chu-Vandermond identity, 
it follows by \eqref{eq:phi_d=1} that 
$$
   \sH_n(x, \k_1,\k_2,n) = (-1)^{x_1} \frac{(\k_2+1)_n}{(\k_1+1)_n} \sQ_n(x_1; \k_1,\k_2,n)
         = \frac{(-1)^{x_2} (\k_1+1)_n}{(\k_1+1)_{x_1}(\k_2+1)_{x_2}}. 
$$
Setting $\sH = \sH_n(\cdot, \k_1,\k_2, n)$ and $N =n$ in \eqref{eq:Psi-n2} and taking the sum over 
$y = (j,n-j)$ for $j = 0,1,\ldots, n$, we see that the left hand side of \eqref{eq:Psi-n2} becomes 
\begin{align*}
 & (\k_1+1)_n \sum_{j=0}^n (-1)^{n-j} \frac{(\k_1+1-x_1)_j (\k_2+x_2+1)_{n-j}}{(k_1+1)_j(k_2+1)_{n-j} j!(n-j)!} \\
    & \quad = \frac{ (\k_1+1)_n  (\k_2+x_1+1)_n (-1)^n}{ (\k_2+1)_n n!}  {}_3 F_2 \left( \begin{matrix} -n,
        \k_1+x_1+1, -n - \k_2, \\  \k_1+1, -k_2- x_2-n \end{matrix}; 1\right)
\end{align*}
where we have used $(a)_{n-j} = (-1)^j (a)_n/ (1-a-n)_j$ in order to write the left hand side as the ${}_3F_2$. Consequently, \eqref{eq:Psi-n2} in this case implies a closed-form for a ${}_3F_2$ evaluated at 1. 
Setting $x_1 = m$ so that $x_2 = n-m$, we state the result as follows. 

\begin{prop}
Let $\k_1> -1$, $\k_2 > -1$ and let $n$ be a positive integer. For $m = 0,1,\ldots, n$, 
\begin{equation}  \label{eq:3F2at1}
 {}_3 F_2 \left( \begin{matrix} -n, \k_1+ m +1, -n - \k_2, \\  \k_1+1, -\k_2+m - 2n \end{matrix}; 1\right)
    = \frac{n!(\k_2+1)_n (-\k_2-2n)_m}{(\k_1+1)_m (\k_2+1)_{2n}}.  
\end{equation}
\end{prop}

This ${}_3 F_2$ is not balanced, so that the identity \eqref{eq:3F2at1} is not a consequence of the 
Pfaff-Saalsch\"utz identity. We do not know if this identity is new. 

\section{Kernels for the Hahn polynomials}
\setcounter{equation}{0} 

Recall that the reproducing kernels for the Jacobi polynomials are denoted by 
$P_n(W_\k; \cdot,\cdot)$ and the reproducing kernels for the Hahn polynomials are denoted by
$\sP_n(\sH_{\k,N}; \cdot,\cdot)$. The generating relation between the Hahn polynomials and the 
Jacobi polynomials extends to their corresponding reproducing kernels.

\begin{lem} 
For $x = (x',x_{d+1})\in \RR^{d+1}$ and $y = (y',y_{d+1}) \in \RR^{d+1}$,
\begin{align} \label{eq:J=Hreprod}
 |x|^N |y|^N P_n \Big(W_\k; \frac{x'}{|x|}, \frac{y'}{|y|} \Big) & = 
  \frac{(-1)^n(|\k|+d+1)_{N+n} }{ (-N)_n(|\k|+d+1)_{N}} \\
    & \times \sum_{|\a| = N} \sum_{|\b|=N} \frac{N!}{\a!}\frac{N!}{\b!}\sP_n (\sH_{\k,N}; \a,\b) x^\a y^\b. \notag
\end{align}
\end{lem}
   
\begin{proof}
From the generating relation \eqref{Hahngenfunc}, it is easy to see that 
\begin{align*}
 |x|^N |y|^N P_n \Big(W_\k; \frac{x'}{|x|}, \frac{y'}{|y|} \Big) & = |x|^N |y|^N \sum_{|\nu| = n} 
    \frac{P_\nu^\k(\frac{x'}{|x|}) P_\nu^\k(\frac{y'}{|y|})}{A_\nu(\k)} \\  
     & = \sum_{|\a| = N}  \frac{N!}{\a!}  \sum_{|\b|=N} \frac{N!}{\b!} 
     \sum_{|\nu| = n} \frac{\sH_\nu(\a; \k, N)\sH_\nu(\a; \k, N) }{A_\nu(\k)} x^\a y^\b. \notag
\end{align*}
By \eqref{eq:B-A}, we can replace $A_\nu(\k)$ by $\sB_\nu(\k,N)$ in the inner sum at the cost of a constant 
that can be pulled outside of the sum, so that the inner sum can be written as a constant multiple of 
$\sP_n(\sH_{\k,N}; \a, \b)$ by \eqref{Hahn-reprod}. This gives the desired identity. 
\end{proof}

This lemma allows us to derive representations for $\sP_n(\sH_\k; \cdot,\cdot)$ from the integral 
representation of $P_n(W_\k; \cdot, \cdot)$ in \eqref{eq:Pclosed-form}. First we need a definition. 

\begin{defn}
Let $\k \in \RR^{d+1}$ with $\k_i > -1$. For $x, y \in \ZZ_N^{d+1}$, define
$\CE_0 (x,y; \k):= 1$ \allowbreak and 
$$
 \CE_n (x,y; \k) := \sum_{\g \in \ZZ_n^{d+1}} \frac{(-x)_\g (-y)_\g}{(\k+1)_\g \g!}, \quad n=1,2,\ldots, N. 
$$
\end{defn}  

The function $\CE_{n}(x,y; \k)$ can be written down explicitly. For example, 
\begin{align*}
   \CE_{1} (x,y; \k) = \sum_{j=1}^d \frac{x_jy_j}{\k_j+1} + \frac{(N-|x|)(N-|y|)}{\k_{d+1}+1}.
\end{align*}
Furthermore, $\CE_{n}(x,y)$ is nonnegative for $x, y \in \ZZ_N^{d+1}$.

Our first representation for $\sP_n(\sH_\k; \cdot,\cdot)$ is an analogue of the integral representation 
\eqref{eq:Pclosed-form} for $P_n(W_\k; \cdot, \cdot)$.

\begin{thm} \label{thm:PnH1} 
For $\k \in \RR^{d+1}$ with $\k_i > -1$, $n= 0,1,\ldots, N$ and $\a, \b\in \ZZ_N^{d+1}$, 
\begin{align}\label{PnH-closed1}
    \sP_n(\sH_{\k,N};\a, \b) =  & \, C_n(\k,N) \sum_{m=0}^N (-1)^m \binom{N}{m} \sQ_n(m; |\k|+d-\tfrac12, -\tfrac12,N) \\
       & \times \sum_{k=0}^m 
    \Big ( \f12 \Big)_{N-k} \frac{(-m)_k}{\binom{N}{k}} \CE_{N-k}(\a,\b; \k), \notag
\end{align}
where 
$$  
   C_n(\k, N) := \frac{(|\k|+d+1)_N(|\k|+d+1)_n(|\k|+d+\f12)_n (|\k|+d+2n) (-N)_n}
       { (|\k|+d+1)_{N+n} (\f12)_n (|\k|+d+n) n! N! (-1)^n}.
$$
\end{thm}

\begin{proof}
By analytic continuation, we only need to establish \eqref{PnH-closed} under the assumption 
that $\k_i > -1/2$ for $1 \le i \le d+1$, which we shall assume in the proof  in order to
use \eqref{eq:Pclosed-form}. 
Setting $\frac{1-t}{1+t} = 2 z -1$ in \eqref{eq:1dHahn}, we obtain
$$
 \frac{ P_n^{(a,b)}(2z-1) }{ P_n^{(a,b)}(1) } = \sum_{m=0}^N\binom{N}{m} \sQ_n(m; a,b,N)(1-z)^m z^{N-m}.
$$
Using this formula in the integral representation \eqref{eq:Pclosed-form} of the kernel $P_n (W_\k; \cdot, \cdot)$, 
we obtain that, for $x, y \in T^d$, 
\begin{align} \label{eq:PnW-int}
  P_n(W_\k;x,y) & = \frac{(|\k|+d+1)_n (|\k|+d+\f12)_n (|\k|+d+2n)}{(\f12)_n n! (|\k|+d+2n)} \\
   & \times \sum_{m=0}^N \binom{N}{m} \sH_n(m; |\k|+d-\tfrac12,-\tfrac12,N) L_m(x,y) \notag
\end{align}
where we have used $P_n^{(a,b)}(1) = (a+1)_n/n!$ and 
$$
   L_m(x,y): = c_\k \int_{[-1,1]^{d+1} }(1-z(t;x,y)^2)^m z(t;x,y)^{2N-2m} \prod_{i=1}^{d+1} (1-t_i^2)^{\k_i-\f12} dt.
$$
Expanding $(1-z(t;x',y')^2)^m$ by the binomial identity, it follows that 
$$
L_m(x,y) =  \sum_{k=0}^m (-1)^{m-k} \binom{m}{k} 
 c_\k \int_{[-1,1]^{d+1} } ( z(t;x,y))^{2N-2k} \prod_{i=1}^{d+1} (1-t_i^2)^{\k_i-\f12} dt.
$$
Next, we apply the multinomial identity to expand $(z(t; x,y))^{2N-2k}$. If a term in the
expansion contains $t^\g$ with at least one component $\g_i$ of $\g$ being an odd integer,
then the integral of the term must be zero. Consequently, this gives 
\begin{align*}
  L_m(x,y) =  &\sum_{k=0}^m (-1)^{m-k} \binom{m}{k} 
  \sum_{|\g| = N-k} \frac{(2N-2k)!}{(2\g)!} X^\g Y^\g \\
    &\times  c_\k \int_{[-1,1]^{d+1} } t^{2\g} \prod_{i=1}^{d+1} (1-t_i^2)^{\k_i-\f12} dt, 
\end{align*}
where $X$ and $Y$ are defined as in \eqref{eq:X}. The last integral can be easily computed to be 
$(\f{\one}2)_\g /(\k+\one)_\g$. Hence, using 
$$
    \frac{(2N-2k)!}{(2 \g)!} = \frac{ (\frac12)_{N-k} (N-k)!}{ (\frac{\one}2)_\g \g!} \quad
    \hbox{and}\quad \binom{m}{k} = \frac{(-1)^k (-m)_k}{k!}, 
$$
we conclude that 
$$
  L_m(x,y) =  (-1)^m  \sum_{k=0}^m \frac{ (\frac12)_{N-k} (N-k)! (-m)_k} {k!} 
         \sum_{|\g| =N-k} \frac{1}{\g! (\k + \one)_\g} X^\g Y^\g.  
$$
Since $x = X/|x|$, using the multinomial identity on $|x|^{N-|\g|}$ shows that 
\begin{align*}
   |x|^N X^\g = |x|^{N-|\g|} x^\g & = \sum_{|\a| = N-|\g|} \frac{(N-|\g|)!}{\a!} x^{\a+\g} \\
     & = \sum_{|\a| = N} \frac{(N-|\g|)!}{(\a - \g)!} x^\a = (N-|\g|)! \sum_{|\a| = N} \frac{ (-1)^{|\g|}(- \a)_\g}{\a!} x^\a.
\end{align*}
Expanding $|y|^N Y^\g$ as a sum over $|\b| =N$ analogously, we then obtain an expansion of 
$|x|^N |y|^N L_m(x,y)$ in terms of $x^\a y^\b$, which is  
\begin{align}\label{eq:Lm}
  |x|^N |y|^N L_m\left(\frac{x}{|x|}, \frac{y}{|y|}\right) &  =  (-1)^m \sum_{|\a| =N} \sum_{|\b|=N} \frac{1}{\a!\b!} \\
      & \times \sum_{k=0}^m 
    \Big ( \f12 \Big)_{N-k} (-m)_k (N-k)! k! \CE_{N-k}(\a,\b;\k) x^\a y^\b. \notag 
\end{align}
From \eqref{eq:PnW-int} and \eqref{eq:Lm}, we obtain another expansion of the left hand side of 
\eqref{eq:J=Hreprod} in $x^\a y^\b$. Comparing it with the right hand side of \eqref{eq:J=Hreprod}
gives a representation of the kernel $\sP_n(\sH_{\k,N};\a,\b)$, which simplifies to \eqref{PnH-closed1}. 
\end{proof}

The right hand side of \eqref{PnH-closed1} can be further simplified to the following closed-form
formula for the reproducing kernel. 

\begin{thm} \label{thm:PnH}
For $\k \in \RR^{d+1}$ with $\k_j > -1$, $n = 0,1,\ldots, N$ and $x,y \in \ZZ_N^{d+1}$, 
\begin{align}\label{PnH-closed}
     \sP_n(\sH_{\k, N}; x,y) = & \frac{(-N)_n (|\k|+d+1)_N (|\k|+d+1)_n (|\k|+d+2n) }{ n!(|\k|+d+1)_{N+n}(|\k|+d+n)} \\
    & \times  \sum_{k=0}^n \frac{(-n)_k (n+|\k|+d)_k}{(-N)_k(-N)_k} \CE_{k} (x,y;\k). \notag
\end{align}
\end{thm}

\begin{proof}
To simplify the double sums in \eqref{PnH-closed1}, we make the following rearrangement 
\begin{align*}
   \sum_{m=0}^N \sum_{k=0}^m a_{k,m}  &  = \sum_{k=0}^N \sum_{m=k}^N a_{k,m}
       = \sum_{k=0}^N \sum_{m=0}^{N-k} a_{k,m+k}  \\
     & = \sum_{k=0}^N \sum_{m=0}^{k} a_{N-k,m+N-k}
            = \sum_{k=0}^N \sum_{m=0}^{k} a_{N-k,m+N}, 
\end{align*}
which shows that 
\begin{align*}
    \sP_n(\sH_{\k,N};x,y) & =  C_n(\k,N) \sum_{k=0}^N \Big ( \f12 \Big)_{k} \frac{1}{\binom{N}{k}} \CE_{k}(\a,\b; \k)\\
      & \times \sum_{m=0}^k  (-1)^{m+N} \binom{N}{m}(-N+m)_{N-k} \sH_n(N-m;|\k|+d-\tfrac12, -\tfrac12, N). \notag 
\end{align*}
Using $(-N+m)_{N-k}  = (-1)^{N-k} (N-m)!/(k-m)!$, $(-N)_k = (-1)^k N!/(N-k)!$ and \eqref{QnQn}, we obtain then
\begin{align*}
    \sP_n(\sH_{\k,N};x,y) & =  C_n(\k,N) \frac{(-1)^{n} (\f12)_n}{(|\k|+d+\f12)_n}
    \sum_{k=0}^N \frac{ N! (\f12)_{k} }{ (-N)_k} \CE_{k}(\a,\b; \k)\\
      & \times \sum_{m=0}^k  (-1)^{m} \binom{k}{m} Q_n(m; |\k|+d-\tfrac12, -\tfrac12, N). \notag 
\end{align*}
Hence, to prove \eqref{PnH-closed}, it suffices to prove that, for $k = 0, 1,\ldots, N$, 
\begin{equation}\label{eq:sumHahn}
 \sum_{m=0}^k  (-1)^{m} \binom{k}{m} Q_n(m; |\k|+d-\tfrac12, -\tfrac12, N) = 
  \frac{(-n)_k (n+|\k|+d)_k k!}{(\f12)_k (-N)_k}. 
\end{equation}
By the definition of the Hahn polynomial in \eqref{eq:1dHahn}, the left hand side of \eqref{eq:sumHahn}
becomes, after changing summation order, 
$$
  \textrm{LHS} = \sum_{j=0}^n \frac{(-n)_j (n+|\k|+d)_j}{(\f12)_j (-N)_j j!} \sum_{m=0}^k \frac{(-k)_m(-m)_j}{m!}.
$$
Since $(-m)_j = 0$ for $m < j$ and $(-m)_j = (-1)^j m!/(m-j)!$, we have 
\begin{align*}
 \sum_{m=0}^k \frac{(-k)_m(-m)_j}{m!} & = \sum_{m=j}^k \frac{(-k)_m (-1)^j}{(m-j)!} \\
   &  = \sum_{m=0}^{k-j} \frac{(-k)_{m+j} (-1)^j}{(m)!} 
     = (-1)^j (-k)_j \sum_{m=0}^{k-j} \frac{(-k+j)_m}{m!} \\
    &   = (-1)^j (-k)_j (1-1)^{k-j} =   (-1)^j (-k)_j \delta_{k,j}.
\end{align*}
Substituting this into $\textrm{LHS}$ and using $(-k)_k = (-1)^k k!$ proves \eqref{eq:sumHahn}.
\end{proof}

The closed-form \eqref{PnH-closed} was derived in \cite[Prop. 3.5]{GS13} using a different method. 

It is known that $\sQ_n(0; a,b N) = 1$. Hence, by its explicit formula in \eqref{eq:Hn-prod}, we easily
deduce that 
$$
  \sH_\nu(0; \k, N) =   (-1)^{|\nu|}  \prod_{j=1}^d  \frac{(\kappa_j+1)_{\nu_j}}{(a_j+1)_{\nu_j}}. 
$$
The following corollary justifies the reason that we call \eqref{PnH-closed} a closed-form formula. 

\begin{cor} \label{cor:2.4}
Let $\k \in \RR^{d+1}$ with $\k_j > -1$, $x' \in \NN_0^d$ and $x =(x', N-|x'|) \in \ZZ_N^{d+1}$. 
Then for $n = 0,1, \ldots, N$, 
\begin{align}\label{PnH-y=0}
      \sP_n(\sH_{\k,N};x,0) = & \frac{(|\k|+d+1)_N (|\k|+d+1)_n (|\k|+d+2n) (-N)_n}{(|\k|+d+1)_{N+n}(|\k|+d+n) n!} \\
      & \times  \sQ_n(N -|x'|; \k_{d+1}, |\k| - \k_{d+1} + d-1, N). \notag
\end{align}
In particular, $ \sP_n(x, 0; \sH_\k, N)$ is a function of $|x'|$ only. 
\end{cor}

\begin{proof} 
If $y = 0$, then $(-y)_\g = 0$ unless $\g_1 = \ldots = \g_d =0$, so that 
$$
   \CE_{k}(x,0; \k) = \frac{ (-N)_k (-N+|x'|)_k}{(\k_d+1)_k k!}. 
$$
It follows that the sum over $k$ in \eqref{PnH-closed} is a ${}_3F_2$ series evaluated at $1$, which is
the Hahn polynomial given in \eqref{PnH-y=0}. 
\end{proof}

The Poisson kernel of the Hahn polynomials is denoted by 
$$
  \Phi_r(\sH_{\k,N};x,y): = \sum_{n=0}^N \sP_n(\sH_{\k,N};x,y)  r^n, \qquad 0 \le r \le 1.
$$
The closed-form formula \eqref{PnH-closed} can be used to show that the Poisson kernel is nonnegative. 

\begin{thm}\label{thm:PoissonH}
Let $\k \in \RR^{d+1}$ with $\k_j > -1$. The Poisson kernel $\Phi_r(\sH_{\k,N};x,y)$ 
is nonnegative if $x,y \in \ZZ_N^{d+1}$ and $0 \le r \le 1$. Furthermore, for $r =1$, 
\begin{equation} \label{eq:PoissonH+}
        \Phi_1(\sH_{\k,N};x,y)= \frac{(|\k|+d+1)_N}{N!} \CE_N(x,y;\k). 
\end{equation}
\end{thm}

\begin{proof}
Using the closed-form formula \eqref{PnH-closed} and changing the summation order, we obtain
\begin{align}\label{eq:PossinH}
 \Phi_r(\sH_{\k,N};x,y) = (|\k|+d+1)_N \sum_{k=0}^N \frac{\CE_k(x,y;\k) }{(-N)_k(-N)_k} 
     \phi_k(r; \k, N),
\end{align}
where 
$$
  \phi_k(r; \k,N) := \sum_{n = k}^N \frac{(|\k|+d+1)_n(|\k|+d+2n) (-N)_n (-n)_k (n+|\k|+d)_k}
      {(|\k|+d+1)_{N+n} n! (\k|+d+n)} r^n, 
$$
If $a$ and $m$ are both nonnegative integer, then $(-a)_m = 0$ if $a< m$, and 
$$
(-1)^m (-a)_m = a(a-1) \ldots (a- m+1) > 0, \quad \hbox{if $a \ge m$}, 
$$
from which it follows immediately that $\CE_k(x,y; \k) \ge 0$ for $x, y \in \ZZ_N^{d+1}$. 
Hence, it is sufficient to prove that $\phi_k(r,\k, N) \ge 0$ for $0 \le r \le 1$. Using $(-n)_k/n! 
= (-1)^k/(n-k)!$ and 
$$
  \frac{(|\k| + d+1)_n (n+|\k|+d)_k}{|\k|+d+n} =  \frac{(|\k| + d)_n (n+|\k|+d)_k}{|\k|+d} = 
   \frac{(|\k| + d)_{n+k}}{|\k|+d}, 
$$
we obtain, after changing the summation index and using $(a)_{n+k} = (a)_k (a+k)_n$, 
\begin{align*}
   \phi_k(r; \k,N)    = &\, \sum_{n = 0}^{N-k}  \frac{ (-1)^k (-N)_{n+k} (|\k|+d)_{n+2k} (|\k|+d+2 k + 2n)}
        {(|\k|+d+1)_{N+n+k} n! (\k|+d)} r^n \\
  = &\, \frac{(-1)^k (-N)_k (|\k|+d+1)_{2k}}{(|\k|+d+1)_{N+k}} \psi(r; |\k|+d+2k, N-k) r^k,
\end{align*}
where $\psi(r; \s, 0): = 1$ and for $m =1,2,\ldots$, 
$$
  \psi(r; \sigma, m) := \sum_{n = 0}^{m}  \frac{(-m)_n (\sigma)_n (\sigma+ 2n)}
        {(\s+m+1)_n n! \sigma } r^n .
$$
It suffices to prove that $\psi(r; \sigma, m) \ge 0$ for $\s > 0$, $m = 1,2,\ldots$ and $0 \le r \le 1$.  
Splitting the sum to two terms in view of the two terms in $\sigma + 2n$, we can write $\psi(r; \sigma, m)$
as a difference of two hypergeometric functions: 
$$
 \psi(r; \sigma, m) =  {}_2 F_1 \left( \begin{matrix} -m, -\s \\
           \s + m +1 \end{matrix}; r \right) - \frac{2 r m}{ \s + m +1} 
        {}_2 F_1 \left( \begin{matrix} -m+1, -\s+1 \\
           \s + m +2 \end{matrix}; r \right). 
$$
Using the integral representation of the hypergeometric functions and integrating by parts,  we can further write
\begin{align*}
 \psi(r; \sigma, m) = \frac{\Gamma(\s+m+1)}{\Gamma(\s)\Gamma(m+1)} 
    &\left( \int_0^1 t^{\s-1}(1-t)^m (1-r t)^m dt \right. \\
     &\quad  \left. - \frac{2 r m}{ \s + m +1}
        \int_0^1 t^{\s-1}(1-t)^m (1-r t)^{m-1} dt \right).
\end{align*}
Integration by parts shows that the first integral is equal to 
\begin{align*}
 \f{m}{\s} \int_0^1 t^{\s}(1-t)^{m-1} (1-r t)^m dt
  + \f{m r}{\s} \int_0^1 t^{\s}(1-t)^{m} (1-r t)^{m-1} dt,
\end{align*}
which implies, since $(1-r t ) - r (1-t) = 1-r$, that 
\begin{align*}
\psi(r; \sigma, m) = &\, \frac{\Gamma(\s+m+1)}{\Gamma(\s)\Gamma(m+1)}  \f{m}{\s}(1-r) 
   \int_0^1 t^{\s}(1-t)^{m-1} (1-r t)^{m-1} dt \\
   = & \, (1-r)  {}_2 F_1 \left( \begin{matrix} -m+1, \s+1 \\
           \s + m +1\end{matrix}; r \right).
\end{align*}
In particular, it shows that $\psi(r; \sigma, m) > 0$ if $0 \le r <1$, which completes the proof that the 
Poisson kernel is nonnegaive. Furthermore, it also implies that $\psi(1; \sigma, m) = 0$ if $m \ge 1$. 
Hence, it follows that 
$$
 \phi_k(1; \k, N) =  \frac{(-1)^N (-N)_N (|\k|+d+1)_{2N}}{(|\k|+d+1)_{N+N}}\delta_{k,N} = N!\delta_{k,N}
$$
from which \eqref{eq:PoissonH+} follows immediately. The proof is completed.  
\end{proof}

For $d =1$, by \eqref{eq:phi_d=1}, the kernel $\Phi_r(\sH_{\k,N};x,y)$ is the Poisson kernel for the 
classical Hahn polynomials of one variable
$$
  \sum_{n=0}^N \frac{\sQ_n(x; \k_1,\k_2, N) \sQ_n(y; \k_1,\k_2, N)}{\sB_n(\k, N)} r^n, \quad x,y =0, 1,\ldots,N,
$$
and this kernel is nonnegative by Theorem \ref{thm:PoissonH} with $d =1$. This appears to be new. A different
kernel, called discrete Poisson kernel, was defined in \cite{GG} by 
$$
  \sum_{n=0}^m \frac{Q_n(x; \k_1,\k_2, N) Q_n(y; \k_1,\k_2, N)}
              {B_n(\k, N)}  \frac{(-m)_n}{(-N)_n},
$$
and shown to be nonnegative for $x,y, m \in \NN$ and $0 \le m,x,y \le N$. 

Finally, the equality in \eqref{eq:PoissonH+} gives an expansion of $\CE_m(x,y; \k)$, which we
state below as a corollary. 

\begin{cor}
For $\k \in \RR^{d+1}$ with $\k_j > -1$ and $N = 0,1,\ldots$, 
$$
\CE_N(x,y;\k) =  \frac{N!}{(|\k|+d+1)_N} \sum_{n=0}^N \sP_n(\sH_{\k,N}; x,y). 
$$
\end{cor}

\section{Kernels for the Krawtchouck polynomials}
\setcounter{equation}{0}

Although the Krawtchouck polynomials $\sK_\nu(\cdot; \rho, N)$ are limits of the Hahn polynomials
$\sH_\nu(\cdot, \k, N)$ as shown in \eqref{Hahn-Kraw}, the Krawtchouck polynomials do not satisfy a
generating function comparable to that of \eqref{Hahngenfunc}. Indeed, taking the limit 
$s \to \infty$ in \eqref{eq:1dHahn} with $a = s p$ and $b = s (1-p)$, the right hand side becomes
a ${}_2F_1$ function that is not a classical orthogonal polynomial. 

The limit process \eqref{Hahn-Kraw}, on the other hand, allows us to derive properties for the
Krawtchouck polynomials from those for the Hahn polynomials. As an example, applying 
the limit relation \eqref{Hahn-Kraw} to Corollary \ref{cor:Hahn-Hahn} gives the following. 

\begin{prop}
For $\rho \in (0,1)^d$ with $|\rho| <1$, $\nu \in \NN_0^d$ with $|\nu| = n$ and $x \in \ZZ_N^{d+1}$, 
\begin{equation} \label{eq:K-monic}
   \sK_\nu(x; \rho, N) = \frac{n!}{(-N)_n} \sum_{|\a|=n} \frac{\sK_\nu(\a; \rho, n)}{\b!}   (-x)_\a. 
\end{equation}
\end{prop}

We deduce results on the reproducing kernels and the Poisson kernels of the Krowtchouck 
polynomials from results in the previous section by the limit process.  

\begin{defn}
Let $\rho \in \RR^d$ with $0 < \rho_i <1$ and $|\rho|<1$, and let $\brho = (\rho, 1-|\rho|) 
\in (0,1)^{d+1}$. For $x,y \in \ZZ_N^{d+1}$ and $n =0,1,\ldots, N$, define $\CF_0(x,y; \k): = 1$ and
$$
   \CF_n (x,y; \rho) := \sum_{|\g| =n } \frac{(-x)_\g (-y)_\g}{\brho^\g \g!}. 
$$
\end{defn}

The function $\CF_n(x,y;\rho)$ can be considered as a limit of $\CE_n(x,y; \k)$. Indeed, 
if we set $\k = t \brho$, then $|\k| = t$ and, for any fixed $a \in \RR$, 
\begin{equation} \label{eq:Ek-Fk}
 \lim_{t \to \infty} (t +a)_k \CE_k (x,y; t \brho, N) = \CF_k(x,y; \rho, N), \quad x, y \in \ZZ_N^{d+1}. 
\end{equation}

Our closed-form formula for the reproducing kernel $\sP_n(\sK_{\rho, N}; \cdot,\cdot)$ is given
in terms of $\CF_k(\cdot, \cdot; \rho)$.

\begin{thm} \label{thm:PnK}
For $\rho \in (0,1)^d$ with $|\rho| < 1$, and $n = 0,1,2,\ldots$, 
\begin{align}\label{PnK-closed}
     \sP_n(\sK_{\rho, N}; x,y) = & \frac{(-N)_n}{n!}
         \sum_{k=0}^n \frac{(-n)_k}{(-N)_k(-N)_k} \CF_k(x,y;\rho).  
\end{align}
\end{thm}

\begin{proof} 
Recall that our inner products are all normalized. The limit relation \eqref{Hahn-Kraw} also extends to 
$$
   \lim_{t \to \infty} t^{-N} \sH_{\nu, t \brho}(x) = \sK_{\rho}(x), \quad x \in \ZZ_N^{d+1}, \quad \hbox{and}\quad
   \lim_{t \to \infty} \sB_{\nu}(t \brho, N) = \sC_{\nu}(\rho, N).
$$
Consequently, it follows that 
$$
  \lim_{t \to \infty} \sP_n( \sH_{t \brho, N}; x,y) = \sP_n(\sK_{\rho, N};x,y), \qquad x,y \in \ZZ_N^{d+1}.  
$$
Taking the same limit in \eqref{PnH-closed} and using \eqref{eq:Ek-Fk} then proves 
\eqref{PnK-closed}. 
\end{proof}

Since $K_n(0; p, N)=1$, it follows readily from \eqref{eq:KrawK} that 
\begin{equation} \label{sK(0)}
  \sK_\nu(0; \rho, N) = (-1)^{|\nu|} \prod_{j=1}^d  \frac{\rho_j^{\nu_j}} {(1-|\brho_j|)^{\nu_j}}.
\end{equation}

\begin{cor} \label{cor:PnKy=0}
Let  $\rho \in (0,1)^d$ with $|\rho| < 1$, and let $x' \in \NN_0^d$ and 
$x =(x', N-|x'|) \in \ZZ_N^{d+1}$. Then for $n = 0,1,2,\ldots$, 
\begin{equation}\label{PnH-closedy=0}
      \sP_n(\sK_{\rho, N}; x,0) =  \frac{(-N)_n}{n!} K_n(N-|x'|;1-|\rho|, N). 
\end{equation}
In particular, $\sP_n(\sK_{\rho, N}; x,0)$ is a function of $|x'|$ only. 
\end{cor}     

\begin{proof}
If $y = 0$, then 
\begin{equation}\label{CF-y=0}
\CF_k(x,0; \rho) =  \frac{(-N)_k (-N+|x|)_k}{(1-|\rho|)^k k!}. 
\end{equation}
Hence, seeing $y = 0$ in \eqref{PnH-closed} gives 
\begin{align*}
 \sP_n(\sK_{\rho, N}; x,0) & = \frac{(-N)_n}{n!} \sum_{k=0}^n \frac{ (-n)_k (-N+|x|)_k}{(-N)_k(1-|\rho|)^k k!}  \\
  &  = \frac{(-N)_n}{n!}  {}_2 F_1 \left( \begin{matrix} -n, -N+|x|\\
           -N \end{matrix}; \frac{1}{1-|\rho|}\right),
\end{align*}
which gives  \eqref{PnH-closedy=0} since the above ${}_2F_1$ is the Krowtchouck polynomial
\eqref{eq:KrawKd=1}.  
\end{proof}

The relative simple form of the reproducing kernel allows us to derive a closed-form formula for
the Poisson kernel of the  Krawtchouck polynomials, which we are not able to do for the Poission
kernel of the Hahn polynomials. For $\rho \in \RR^d$ with $0< \rho_i < 1$ and $|\rho| <1$,
the Poisson kernel is defined by 
$$
  \Phi_r(\sK_{\rho, N}; x,y) := \sum_{n=0}^N  \sP_n(\sK_{\rho, N}; x,y) r^n, \qquad 0 \le r < 1. 
$$

\begin{thm}\label{thm:Poisson}
For  $\rho \in (0,1)^d$ with $|\rho| < 1$, and $x,y \in \ZZ_N^{d+1}$, 
\begin{equation}\label{Poisson}
  \Phi_r(\sK_{\rho, N}; x,y)  = \sum_{k=0}^N \frac{(-1)^k}{(-N)_k} \CF_k(x,y; \rho) r^k (1-r)^{N-k}. 
\end{equation}
In particular, the kernel $ \Phi_r(x,y; \sK_\rho, N)$ is nonnegative if $x, y \in \ZZ_N^{d+1}$ 
and $ 0 \le r \le 1$.
\end{thm}

\begin{proof} 
Using \eqref{PnH-closed} and changing summation order, we obtain that
\begin{align*}
  \Phi_r(\sK_{\rho, N}; x,y) = \sum_{k=0}^N\frac{\CF_k(x,y;\rho)}{(-N)_k (-N)_k} 
      \sum_{n=k}^N \frac{(-N)_n (-n)_k}{n!} r^n. 
\end{align*}
Using $(-n)_k/k! = (-1)^k /(n-k)!$, the summation over $n$ can be summed up as follows: 
\begin{align*}
    & \sum_{n=k}^N \frac{(-N)_n (-n)_k}{n!} r^n   =  \sum_{n=0}^{N-k} \frac{ (-N)_{n+k}}{n!} r^{n+k} \\
      & = (-1)^k (-N)_k r^k \sum_{n=0}^{N-k} \frac{(-N+k)_{n}}{n!} r^{n} =
          (-1)^k (-N)_k r^k (1-r)^{N-k},
\end{align*}
which proves \eqref{Poisson}. 
\end{proof}

In the case $y = 0$ in \eqref{Poisson}, the Poisson kernel can be further summed up. 

\begin{cor}\label{cor:Poisson-y=0}
Let  $\rho \in (0,1)^d$ with $|\rho| < 1$. For $x \in \ZZ_N^{d+1}$ and $0<r<1$, 
\begin{equation}\label{eq:Pos-y=0}
  \Phi_r(\sK_{\rho, N}; x,0)  = (1-r)^{|x|} \Big(1+ \frac{r |\rho|}{1-|\rho|} \Big)^{N-|x|}.
\end{equation}
\end{cor}

\begin{proof}
Using \eqref{CF-y=0}, we see that 
\begin{align*}
 \Phi_r(\sK_{\rho, N}; x,0) &  =  \sum_{k=0}^N\frac{(-1)^k (-N+|x|)_k}{k! (1-|\rho|)^k}  (-1)^k r^k (1-r)^{N-k} \\
   &   =  (1-r)^N \left(1- \frac{ -r}{(1-r)(1-|\rho|)} \right)^{N-|x|},
\end{align*}
which simplifies to \eqref{eq:Pos-y=0}, where we have used the fact that $(-N+|x|)_k = 0$ if $k > N-|x|$ to 
write the sum over $k$ as a summable infinite series. 
\end{proof}

In the case of $d=1$, it follows from \eqref{eq:KrawK} and \eqref{KnKn} that 
$$
  \sK_n(x;\rho, N) = \frac{(-1)^n \rho^n}{(1-\rho)^n} \sK_n(x_1; \rho, N), \quad x =(x_1,x_2)\in \ZZ_N^2,
$$
where $0< \rho < 1$ and the right hand side is the classical Krawtchouck polynomial in \eqref{eq:KrawKd=1}. 
Using \eqref{sK(0)} and the value of $\sC_\nu(\rho,N)$, the identity in \eqref{eq:Pos-y=0} becomes, when 
writing explicitly, 
\begin{align*} 
(1-r)^{|x|} \Big(1+ \frac{r \rho}{1-|\rho|} \Big)^{N-|x|}  = \sum_{|\nu| \le N} 
  (-N)_{|\nu|} \prod_{j=1}^d \frac{(1- |\brho_j|)^{\nu_{j+1}}}{\nu_j!} \sK_\nu(x; \rho, N) r^{|\nu|}.
\end{align*}
For $d =1$, the above identity arees with the generating function
$$
   \left(1-\frac{1-\rho}{\rho} t \right)^x (1+ t)^{N-x} = \sum_{n=0}^N \binom{N}{n} \sK_n(x; \rho, N)t^n,
$$
of the Krowtchouck polynomials when $t = r \rho /(1-\rho)$. 

Finally, setting $r =1$ in \eqref{Poisson} gives us the orthogonal expansion of $\CF_N(x,y;\rho)$.

\begin{cor}
For $N =0,1,\ldots$, 
$$
    \f{1}{N!} \CF_N (x,y; \rho) =  \Phi_{1}(x,y; \sK_\rho, N) =  \sum_{n=0}^N \sP_n(x,y; \sH_\k, N). 
$$
\end{cor}

\enddocument